\documentclass{aic}

\usepackage{tikz}
\usepackage{xifthen}
\usepackage{subfig}
\usetikzlibrary{calc,positioning,decorations.pathmorphing,decorations.pathreplacing}
\tikzset{emp/.style={double distance = 0.3ex}}
\tikzset{oriented/.style={->,shorten >= 1.5pt}}
\usetikzlibrary{calc,decorations.markings,decorations.pathmorphing,arrows,external,decorations.pathreplacing,arrows.meta,calc,decorations.markings,decorations.pathmorphing,shapes}

\newcommand{\defi}[1]{\emph{#1}}
\DeclareMathOperator{\bigoh}{\mathrm{O}}

\newtheorem{theorem}{Theorem}
\newtheorem{problem}[theorem]{Problem}
\newtheorem{conjecture}[theorem]{Conjecture}

\newtheorem{lemma}[theorem]{Lemma}

\aicAUTHORdetails{%
  title = {Separating the Edges of a Graph by a Linear Number of Paths}, 
  author = {Marthe Bonamy, F\'abio Botler, François Dross, T\'assio Naia, and Jozef Skokan},
  plaintextauthor = {Marthe Bonamy, Fabio Botler, François Dross, Tassio Naia, and Jozef Skokan},
  plaintexttitle = {Separating the Edges of a Graph by a Linear Number of Paths}, 
  runningtitle = {Separating the Edges of a Graph by a Linear Number of Paths},
  runningauthor = {M. Bonamy, F. Botler, F. Dross, T. Naia, and J. Skokan},
  copyrightauthor = {M. Bonamy, F. Botler, F. Dross, T. Naia, and J. Skokan},
}   

\aicEDITORdetails{%
   year={2023},
   number={6},
   received={20 February 2023},   
   published={27 October 2023},  
   doi={10.19086/aic.2023.6},      
}   

\begin{document}

\begin{frontmatter}[classification=text]


\title{Separating the Edges of a Graph by a Linear Number of Paths\titlefootnote{
This research has been partially supported by Coordena\c cão de Aperfei\c coamento
de Pessoal de N\'\i vel Superior -- Brasil -- CAPES -- Finance Code 001.
M. Bonamy is supported by ANR project DISTANCIA {(\small Metric Graph Theory, ANR-17-CE40-0015)}.
F.~Botler is supported by CNPq {(\small 304315/2022-2)},
FAPERJ {(\small 211.305/2019 and 201.334/2022)}
and CAPES-PRINT {(\small 88887.695773/2022-00)}.
T.~Naia is supported by CNPq {(\small 201114/2014-3)} and FAPESP {(\small 2019/04375-5, 2019/13364-7, 2020/16570-4)}.
FAPERJ and FAPESP are, respectively, Research Foundations of
Rio de Janeiro and S\~ao Paulo.  CNPq is the National
Council for Scientific and Technological Development of Brazil.
}}

\author[bonamy]{Marthe Bonamy}
\author[botler]{F\'abio Botler}
\author[dross]{François Dross}
\author[naia]{T\'assio Naia}
\author[skokan]{Jozef Skokan}

\begin{abstract}
Recently, Letzter proved that any graph of order~\(n\)
contains a collection~$\mathcal{P}$ of \(\bigoh(n\log^\star n)\)
paths with the following property:
for all distinct edges $e$ and~$f$ there exists a path in $\mathcal{P}$
which contains \(e\) but not~\(f\).
We improve this upper bound to~\(19 n\), thus answering a question of G.O.H. Katona and
confirming a conjecture independently posed
by~Balogh, Csaba, Martin, and Pluh\'ar and
by~Falgas-Ravry, Kittipassorn, Kor\'andi, Letzter,  and Narayanan. Our proof is elementary and self-contained.
\end{abstract}
\end{frontmatter}

A path \defi{separates} an edge $e$ from an edge $f$ if it contains the former and not the latter. How many paths do we need to separate any edge from any other?

This question was first asked by G.O.H. Katona in 2013 (see~\cite{falgas2013separating}), in line with the general study of separating systems initiated by R\'enyi in 1961~\cite{renyi1961random}. Note that the roles of $e$ and~$f$ are not symmetric. There are in fact two variants of the problem, which we clarify now. A~collection $\mathcal{P}$ of paths in a graph $G$ is a \defi{strongly-separating path system} of $G$ if, for any two distinct edges $e$ and $f$ of $G$, there exist paths $P_e$ and~$P_f$ in~$\mathcal{P}$ such that $e \in P_e$, $e \not\in P_f$ and $f\in P_f$, $f \not\in P_e$.
The collection $\mathcal{P}$ is a \defi{weakly-separating path system} of $G$ if for any two distinct edges $e$ and $f$, there is a path $P \in \mathcal{P}$ that contains one and not the other.
The \defi{size} of a (weakly or strongly) separating path system is the number of paths in the collection.
We recommend the excellent introduction in~\cite{letzter2022separating} for a more complete description of the history of the problem.
We are interested in the following conjecture.

\begin{conjecture}[\cite{balogh2016path,falgas2013separating}]\label{conj:lin-sep}
Every graph on $n$ vertices admits a strongly-separating path system of size $\bigoh(n)$.
\end{conjecture}

This conjecture was formulated independently by~Falgas-Ravry, Kittipassorn, Kor\'andi, Letzter, and Narayanan~\cite{falgas2013separating} for weakly-separating path systems,
and later strengthened by~Balogh, Csaba, Martin, and Pluh\'ar~\cite{balogh2016path} to strongly-separating path systems.
Very recently, Letzter~\cite{letzter2022separating} made
major progress toward Conjecture~\ref{conj:lin-sep} by proving
that every graph on \(n\) vertices admits a strongly-separating
system with $\bigoh(n\cdot \log^\star n)$ paths. A key component
in~\cite{letzter2022separating} is the use of sublinear
expanders, motivated by the recent work of Buci\'c and
Montgomery~\cite{bucic2022towards}.
 In this note we obtain a linear bound, thus confirming Conjecture~\ref{conj:lin-sep}.

\begin{theorem}\label{thm:main}
  Every graph on~$n$ vertices	admits a strongly-separating path system of size~\(19n\).
\end{theorem}

Our proof uses Pósa rotation--extension tools~\cite{posa1976hamiltonian}
as presented by Brandt, Broersma, Diestel, and Kriesell~\cite{brandt2006global},
and is inspired by  Lemma~2.2 in Conlon, Fox, and Sudakov~\cite{conlon2014cycle}.
In essence, we reduce the general problem to
\defi{traceable graphs}, i.e., graphs with a path that spans all vertices, known as a Hamiltonian path.

Before delving into technical arguments, we point out that the upper bound of $19n$ is likely far from optimal.
Balogh et al.~\cite[Theorem 10]{balogh2016path} showed that the complete bipartite graph \(K_{\varepsilon n,(1-\varepsilon)n}\)
cannot be strongly separated by fewer than \(2(1-2\varepsilon)n\)~paths.
However, we are not aware of any graph with a larger lower bound. Therefore, the following might be true.
\begin{problem}
  Does every graph on $n$ vertices admit a strongly-separating path system of size~$2n$\,?
\end{problem}
Determining the exact constant is challenging even for cliques~\cite{wickes}.
In our initial attempts to solve Conjecture~\ref{conj:lin-sep}, we stumbled upon this question:

\begin{problem}\label{prob:rainbow}
Does every properly edge-colored graph \(G\) admit
$\bigoh(|V(G)|)$ rainbow paths that cover~$E(G)$\,?
\end{problem}

In this context, a rainbow path is a path that does not contain two edges of the same color.
To the best of our knowledge, this question has not been stated anywhere else.
  We pose this question here, since we believe it to be natural and interesting in its own right.
Problem~\ref{prob:rainbow}, if answered affirmatively, would yield an alternative proof of Conjecture~\ref{conj:lin-sep}:
start with a linear path decomposition (e.g., with Theorem~\ref{thm:dean-kouider}),
decompose each path into two matchings, and define two graphs each containing precisely one of the matchings of each path.
The matchings form a proper edge coloring of each such graph,
and the two linear rainbow path decompositions of these graph together with the original path decomposition form a strongly separating path system.  We remark that we can verify a version of Problem~\ref{prob:rainbow}  where ``paths'' is replaced by ``trails''.

\subsection*{Proof of Theorem~\ref{thm:main}}

From now on, every separating path system is a strongly-separating path system.
We use the following standard notation.
Given a graph \(G\), and  a set \(S\subseteq V(G)\),
we denote by \(N_G(S)\) the set of vertices \emph{not} in $S$ adjacent in \(G\) to some vertex in \(S\).
We omit subscripts when clear from the context.

\medskip
\noindent\textbf{Pósa rotation-extension.}
Given a graph $G$ and vertices $u, v$ in \(G\), let \(P = u \cdots v\) be a path from $u$ to $v$.
If \(x\in V(P)\) is a neighbor of \(u\) in \(G\)
and \(x^-\) is the vertex preceding \(x\) in~\(P\),
then \(P' = P-xx^- + ux\) is a path in \(G\) for which \(V(P') = V(P)\).
We say that \(P'\) \defi{has been obtained from} \(P\) by an
\defi{elementary exchange} fixing~\(v\) (see Figure~\ref{fig:rotation}).
A path obtained from~\(P\) by a (possibly empty) sequence of elementary exchanges fixing~\(v\)
is said to be a path \defi{derived} from~\(P\).
The set of endvertices of paths derived from~\(P\) that are distinct from~\(v\) is denoted by~\(S_v(P)\).
Since all paths derived from~\(P\) have the same vertex set as~\(P\), we have~\(S_v(P) \subseteq V(P)\).
When \(P\) is a longest path ending at \(v\), we obtain the following.

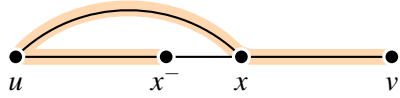
\begin{figure}[t]
          \begin{center}
            \tikzset{highlight/.style={orange!75!white,line width=6pt,line cap=round}}
\tikzset{highlight2/.style={tngreen!70!yellow!60,line width=11pt,line cap=round}}
\tikzset{psline/.style={decorate,decoration={amplitude=0.3mm,segment length=2mm,post length=1.5mm,pre length=1.5mm,coil,aspect=0}}}
\tikzset{plainlin/.style={line width=1pt}}
\def\shortendistance{5pt}
\tikzset{plainlin shorten/.style={line width=1pt,shorten >=\shortendistance,shorten <=\shortendistance}}

\tikzset{lin shorten/.style={shorten >=\shortendistance,shorten <=\shortendistance,line width=1pt,decoration={markings, mark=at position 0.5*\pgfdecoratedpathlength+1.8pt with {\arrow{angle 90}}}, postaction={decorate}}}

\newcommand{\vertexat}[1]{\fill [very thick,draw=white,fill=black] #1 circle (3pt);}

\begin{tikzpicture}

    \coordinate  (u)   at  (0,  1.5);
    \coordinate  (x-)  at  (2,  1.5);
    \coordinate  (x)   at  (3,  1.5);
    \coordinate  (v)   at  (5,  1.5);

    \begin{scope}[opacity=.4]
      \draw [highlight,rounded corners] (x-) to (u) to [out= 45, in=135] (x)
      to  (v);
    \end{scope}

    \draw [thick] (u) to (x-) to (x) to (v);
    \draw [thick] (u) to [out=45,in=135] (x);

    \foreach \nn in {u,x-,x,v} {
      \vertexat{(\nn)};
    }

    \node [anchor=north] at (u) {$u\vphantom{x^-}$};
    \node [anchor=north] at (x-) {$x^-$};
    \node [anchor=north] at (x) {$x\vphantom{x^-}$};
    \node [anchor=north] at (v) {$v\vphantom{x^-}$};
\end{tikzpicture}
            \caption{a path (highlighted)
              obtained by an elementary exchange fixing~\(v\).
            }\label{fig:rotation}
          \end{center}
\end{figure}

\begin{lemma}[\cite{brandt2006global}]\label{lemma:brandt}
  Let  \(P = u \cdots v\) be a longest path of a graph $G$
  and let \(S = S_v(P)\). Then \(|N_G(S)| \leq 2|S|\).
\end{lemma}

\begin{proof}[{Proof rephrased from~{\cite[Lemma 2.6]{brandt2006global}}}]
  Since the maximum degree of \(P\)~is~\(2\), it suffices to show  that \(N_G(S) \subseteq N_P(S)\).
  Let \(x\in S\) and \(y\in N_G(x)\setminus S\) be given.
  As \(x\in S\) there is a longest path \(Q = x \cdots v\) derived from \(P\).
  Note that, by the maximality of~\(Q\), we have $y\in V(Q)$.
  Let~\(z\) denote the predecessor of \(y\) in~\(Q\),
  and note that \(Q + xy - yz\) is obtained from~\(Q\) by an elementary exchange.
  Now, consider the first edge \(e\) incident to~\(y\) that is removed by
  an elementary exchange in the ``derivation sequence'' from \(P\) to~\(Q + xy - yz\),
  and note that \(e\in E(P)\).
  Note that when an elementary exchange removes an edge~\(ab\) of a path derived from~\(P\),
  then \(a\)~or~\(b\) belongs to~\(S\).
  Therefore one of the vertices of~$e$ is in~$S$. Since $y \not\in S$ and~$e \in E(P)$,
  we obtain \(y\in N_P(S)\), as~desired.
\end{proof}

Our argument requires a lemma stating that every $n$-vertex graph admits an edge-cover
by $\bigoh(n)$~paths.
  While elementary arguments yield a cover by at most \(6n\)~paths
  (as we discuss at the end of the paper),
  we optimize the multiplicative constant in Theorem~\ref{thm:main}
  using a result of Dean and~Kouider~\cite{dean2000gallai}
  that strengthens a well-known result of Lovász~\cite{lovasz1968covering}.

\begin{theorem}[\cite{dean2000gallai}]\label{thm:dean-kouider}
  Every graph~$G$ contains at most \(2|V(G)|/3\) edge-disjoint~paths covering~$E(G)$.
\end{theorem}

Let \(H\) and \(G\) be (not necessarily disjoint) graphs.
We say that a set of paths \(\mathcal{P}\) in \(H \cup G\) \defi{separates}~\(H\) from~\(G\)
if for each pair \((e,f) \in E(H) \times E(G)\) with \(e \ne f\)
there is a path in \(\mathcal{P}\) that contains \(e\) and does not contain \(f\).
Now, we are able to prove our result.

\begin{proof}[Proof of Theorem~\ref{thm:main}]
We proceed by induction on $n$.
Let \(G\) be a graph with \(n\) vertices.
If~\(G\)~is empty, the result trivially holds.
If not, we consider $P$ and $S$ as in Lemma~\ref{lemma:brandt}.
Let \(H\) be the subgraph of \(G\) induced by the edges with at least one vertex in~\(S\),
set \(n' = |V(H)|\),
and note that \(n' \leq 3|S|\).
Let \(G' = G \setminus S\) and note that \(G\) is the (edge) disjoint union of \(G'\) and \(H\).
Note that, since \(G\) is not empty, \(S\) is not empty.
By~the induction hypothesis, \(G'\) admits a separating path system \(\mathcal{Q}'\)
of size \(19(n-|S|)\).
Since \(\mathcal{Q}'\) covers~\(G'\), the family~\(\mathcal{Q}'\) separates  \(G'\) from \(H\).
In what follows, we construct a set of paths \(\mathcal{P}\) that separates
\(H\) from \(G\).
Moreover, we obtain such \(\mathcal{P}\) so that \(|\mathcal{P}| \leq 19|S|\),
and hence \(\mathcal{Q}' \cup \mathcal{P}\)
is a separating path system of \(G\)
with cardinality at most \(19(n-|S|) + 19 |S| = 19n\) as desired.

First, let \(P_S = E(P) \cap E(H)\) be the set of edges of \(P\) having at least one vertex in \(S\) (see Figure~\ref{fig:NS}), and let \(\mathcal{P}_S\) be the set of paths each
consisting of a single edge in \(P_S\), i.e., \(\mathcal{P}_S = \{(e,\{e\}) : e \in P_S\}\).
Now, let \(\mathcal{D}\) be the path decomposition of \(H' = H\setminus P_S\) given by Theorem~\ref{thm:dean-kouider}.
Note that \(\mathcal{P}_S \cup \mathcal{D}\)
has size at most \(2|S| + 2n'/3 \leq 4|S|\) and separates
(i)  \(H\) from \(G'\);
(ii)  \(H'\) from \(P\);
and (iii) \(P_S\) from \(G\).
It remains to create a set of at most \(5n' \leq 15|S|\) paths that separates \(H'\) from itself  (see Figure~\ref{fig:NS}).
\begin{figure}
  \vspace*{-1.5cm}
  \begin{center}
    \hfill
    \begin{minipage}{.48\textwidth}
      \tikzset{highlight/.style={orange!75!white,line width=6pt,line cap=round}}
\tikzset{highlight2/.style={tngreen!70!yellow!60,line width=11pt,line cap=round}}
\tikzset{psline/.style={decorate,decoration={amplitude=0.3mm,segment length=2mm,post length=1.5mm,pre length=1.5mm,coil,aspect=0}}}
\tikzset{plainlin/.style={line width=1pt}}
\def\shortendistance{5pt}
\tikzset{plainlin shorten/.style={line width=1pt,shorten >=\shortendistance,shorten <=\shortendistance}}

\tikzset{lin shorten/.style={shorten >=\shortendistance,shorten <=\shortendistance,line width=1pt,decoration={markings, mark=at position 0.5*\pgfdecoratedpathlength+1.8pt with {\arrow{angle 90}}}, postaction={decorate}}}

\usetikzlibrary{calc,decorations.markings,decorations.pathmorphing,arrows,external,decorations.pathreplacing,arrows.meta,calc,decorations.markings,decorations.pathmorphing,shapes}

\newcommand{\vertexat}[1]{\fill [very thick,draw=white,fill=black] #1 circle (3pt);}

\definecolor{tnred}{RGB}{     255, 63,   63}
\definecolor{tnorange}{RGB}{  220, 140,   6}
\definecolor{tnblue}{RGB}{     10, 154, 215}
\definecolor{tngreen}{RGB}{   114, 165,  10}
\definecolor{tnmagenta}{RGB}{ 215,  10, 154}
\definecolor{tnbg}{RGB}{252,252,252}
\definecolor{tnfg}{RGB}{5,8,12}

\begin{tikzpicture}

    \node at (1,4.3) {$S$};
    \node at (3,4.3) {$N(S)$};

    \draw [dotted,thick,gray!50] (2,-1.5) -- (2,5);
    \draw [dotted,thick,gray!50] (4,-1.5) -- (4,5);

\def\midx{3}
    \coordinate (a) at ($(5,-1)$);
    \coordinate (b) at ($(\midx,-1)$);
    \coordinate (c) at ($(1, -1)$);
    \coordinate (d) at ($(1, 0)$);
    \coordinate (e) at ($(\midx, 0)$);
    \coordinate (f) at ($(5, .5)$);
    \coordinate (g) at ($(\midx, 1)$);
    \coordinate (h) at ($(1, 1)$);
    \coordinate (i) at ($(0,1.5)$);
    \coordinate (j) at ($(1,2)$);
    \coordinate (k) at ($(\midx,2)$);
    \coordinate (l) at ($(5,2.5)$);
    \coordinate (m) at ($(\midx,3)$);
    \coordinate (n) at ($(5,3.5)$);

    \draw [plainlin] (a) -- (b);
    \draw [dashdotted, very thick,red] (b) -- (c);
    \draw [dashdotted, very thick,red] (c) -- (d);
    \draw [dashdotted, very thick,red] (d) -- (e);
    \draw [plainlin]
    (d) -- (h);
    \draw [plainlin]
    (d) -- (i);
    \draw [plainlin]
    (d) -- (g);
    \draw [plainlin] (e) -- (f);
    \draw [plainlin] (f) -- (g);
    \draw [dashdotted, very thick,red] (g) -- (h);
    \draw [dashdotted, very thick,red] (h) -- (i);
    \draw [plainlin]
    (h) -- (j);
    \draw [dashdotted, very thick,red] (i) -- (j);
    \draw [dashdotted, very thick,red] (j) -- (k);
    \draw [plainlin]
    (j) -- (m);
    \draw [plainlin] (k) -- (l);
    \draw [plainlin] (l) -- (m);
    \draw [plainlin] (m) -- (n);

    \foreach \nn in {a,b,c,d,e,f,g,h,i,j,k,l,m,n} {
      \vertexat{(\nn)}
    }

    \draw [highlight2,rounded corners,line cap=round,opacity=0.3] (a) to (b) to (c) to (d) to (e) to (f) to(g) to (h) to (i) to (j) to (k) to (l) to (m) to (n);;

\end{tikzpicture}
    \end{minipage}
    \hfill
    \begin{minipage}{.48\textwidth}
      \vspace*{1.5cm}
      \definecolor{tnred}{RGB}{     255, 63,   63}
\definecolor{tnorange}{RGB}{  220, 140,   6}
\definecolor{tnblue}{RGB}{     10, 154, 215}
\definecolor{tngreen}{RGB}{   114, 165,  10}
\definecolor{tnmagenta}{RGB}{ 215,  10, 154}
\definecolor{tnbg}{RGB}{252,252,252}
\definecolor{tnfg}{RGB}{5,8,12}

\tikzset{highlight/.style={orange!75!white,line width=6pt,line cap=round}}
\tikzset{highlight2/.style={tngreen!70!yellow!60,line width=11pt,line cap=round}}
\tikzset{psline/.style={decorate,decoration={amplitude=0.3mm,segment length=2mm,post length=1.5mm,pre length=1.5mm,coil,aspect=0}}}
\tikzset{plainlin/.style={line width=1pt}}
\def\shortendistance{5pt}
\tikzset{plainlin shorten/.style={line width=1pt,shorten >=\shortendistance,shorten <=\shortendistance}}

\tikzset{lin shorten/.style={shorten >=\shortendistance,shorten <=\shortendistance,line width=1pt,decoration={markings, mark=at position 0.5*\pgfdecoratedpathlength+1.8pt with {\arrow{angle 90}}}, postaction={decorate}}}
\tikzsetnextfilename{fig-separation}

\newcommand{\vertexat}[1]{\fill [very thick,draw=white,fill=black] #1 circle (3pt);}

\begin{tikzpicture}
  \node[shape=circle,fill=gray!40,minimum size=1.4cm] (H')  at (-2, 0)    {$H'$};
  \node[shape=circle,fill=gray!40,minimum size=1.4cm] (G')  at ( 2, 0)    {$G'$};
  \node[shape=circle,fill=gray!40,minimum size=1.4cm] (PS) at ( 0,-3) {$P_S$};

   \begin{scope}[->, shorten >=2pt,thick]
     \draw (H') .. controls +(180:1.5cm) and +(110:1.5cm) .. node [midway,xshift=.5cm,yshift=.7cm] {$P_k$'s and $Q_k$'s} (H');

    \draw (G') .. controls +(80:1.5cm) and +(0:1.5cm) .. node [midway,xshift=.25cm,yshift=.25cm] {$\mathcal{Q}'$} (G');

    \draw (PS) .. controls +(-115:1.5cm) and +(-75:1.5cm) .. node [midway,xshift=0.1cm,yshift=-.3cm] {$\mathcal{P}_S$} (PS);

    \draw (H') .. controls +(40:1.5cm) and +(150:1.5cm) ..  node [pos=.3,yshift=-.3cm] {$\mathcal{D}$}  (G');
    \draw (G') .. controls +(220:1.5cm) and +(-30:1.5cm) ..  node [pos=.3,yshift=.3cm] {$\mathcal{Q}'$}  (H');

    \draw (H') .. controls +(-110:2cm) and +(140:2cm) ..  node [pos=.45,yshift=.3cm,xshift=.2cm] {$\mathcal{D}$}  (PS);
    \draw (PS) .. controls +(170:2.2cm) and +(-130:3cm) ..  node [pos=.2,yshift=-.3cm] {$\mathcal{P}_S$}  (H');

     \draw (G') .. controls +(-70:2cm) and +(40:2cm) ..  node [pos=.45,yshift=.3cm,xshift=-.2cm] {$\mathcal{Q}'$}  (PS);
     \draw (PS) .. controls +(10:2.2cm) and +(-50:3cm) ..  node [pos=.2,yshift=-.3cm] {$\mathcal{P}_S$}  (G');


 \end{scope}

\end{tikzpicture}
    \end{minipage}
    \vspace*{-1.3\baselineskip}
    \caption{Left: a set $S$ in a traceable graph and its neighborhood~$N(S)$;
              a Hamiltonian path is highlighted, dashed red edges are the edges in~$P_S$. Right: paths that separate subgraphs of~$G$, where $A\stackrel{\alpha}{\to}B$ indicates that $\alpha$ separates $A$ from~$B$ (for instance, $\mathcal{Q'}$ separates $G'$ from $H'$).}\label{fig:NS}
  \end{center}
\end{figure}
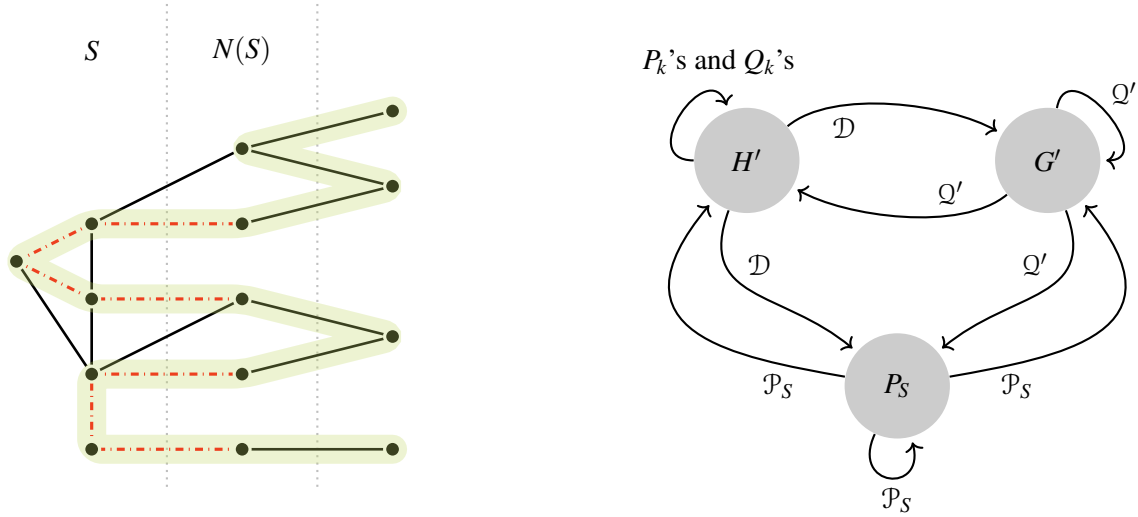

Let us write \(P = u_1 \cdots u_m\),
and let \(v_1,\ldots,v_{n'}\) be the vertices of \(H\) in the order that they appear in \(P\),
i.e., so that if \(v_i = u_{s_i}\) and \(v_j = u_{s_j}\), then \(i < j\) if and only if \(s_i < s_j\).
In what follows, each edge \(v_iv_j\) is written so that \(i < j\).
We define $5n'$ sets of edges as follows:
	for \(k\in [2n']\) let \(M_k = \{v_iv_j \in E(H') : i+j = k\}\),
	and for \(k \in [3n']\) let \(N_k = \{v_iv_j \in E(H') : i+2j = k\}\).
        We shall build a family of~$5n'$ paths, covering each of the $M_k$ and~$N_k$.
        (Strictly speaking, we need fewer than $5n'$ paths, since some $M_k$ and $N_k$
        are empty.
        However, as far as we can tell, this only leads to marginal improvements.)
Note that \(\mathcal{M} = \{M_1,\ldots,M_{2n}\}\) (resp.~\(\mathcal{N} = \{N_1,\ldots,N_{3n}\}\))  partitions \(E(H')\).

        \begin{figure}
          \begin{center}
              \includegraphics{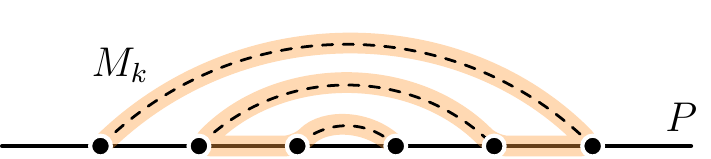}
            \caption{Set of comparable edges $M_k$ (dashed) and a highlighted  path $P_k$ joining them.}\label{fig:nested}
          \end{center}
          \end{figure}

We observe that for any $k$, the edges in $M_k$ are pairwise comparable (similarly for~$N_k$), in the sense that given $v_iv_j, v_{i'}v_{j'} \in M_k$ with $i>i'$ we have $j'>j$, and in fact $j' > j > i > i'$,
i.e.,
each \(M_k\) and \(N_k\) is a set of~``nested'' edges with respect to \(P\).
Moreover, the set $M_k$ is a subset of $\{v_1v_{k-1},\ldots,v_{\frac{k}2-1}v_{\frac{k}2+1}\}$, with the last indices tweaked to fit the parity of $k$. In the same manner $N_k\subseteq\{v_1v_{\frac{k-1}2},\ldots,v_{\frac{k}3-2}v_{\frac{k}3+1}\}$, again with tweaks depending on the value of $k \bmod 6$.
Now, for each \(k \in [2 n']\), we find a path \(P_k\)
using all edges of $M_k$ and (possibly) some edges in $P$ (see Figure~\ref{fig:nested}).
Let \(M_k = \{x_1y_1,\ldots,x_sy_s\}\) where the \(x_i\)'s are chosen in increasing order on \(P\) (hence the $y_i$'s in decreasing order).
By induction on \(i\) from \(s\) down to \(1\), there is a path \(R_i\) in \(P_i \cup \{x_iy_i,\ldots,x_sy_s\}\) that contains \(x_iy_i,\ldots,x_sy_s\) and ends in either \(x_i\) or \(y_i\),
where \(P_i\) denotes the subpath of \(P\) between \(x_i\) and \(y_i\).
Therefore, \(P_k = R_1\) is the desired path.
Analogously, for~\(k \in [3 n']\), we obtain a path \(Q_k\) that contains \(N_k\) and some edges of \(P\).

We claim that the collection of $P_k$'s and $Q_k$'s
together separates each pair of edges in~\(H'\).
Indeed, let $v_iv_j$ and $v_{i'}v_{j'}$ be two edges of \(H'\).
Each such edge belongs to exactly one $M_k$ and exactly one $N_k$.
If~they belong to different \(M_k\)'s, they belong to different $P_k$'s and are separated.
Similarly, if~they belong to different $N_k$'s. Therefore, we may assume that $i+j=i'+j'$ and $i+2 j=i'+2 j'$. This immediately yields $j=j'$ and $i=i'$,
as desired. Therefore, any two distinct edges in $H'$ are separated by this collection, which involves $5n'$ paths.
This concludes the proof.
\end{proof}

\subsection*{Discussion and further work}
As mentioned in the introduction, the best possible constant factor in Theorem~\ref{thm:main} might well be $2$. Here we obtain $19$, which can still be reduced a little with some care\footnote{For example, only applying Theorem~\ref{thm:dean-kouider} once for the whole graph, for all edges not in any $P$, would give an improvement of~$1.3333$. We decided against implementing this and other small improvements, as they do not seem worth the loss of clarity and simplicity.
}. While partial progress would only be of limited interest, we highlight here straightforward candidates for improvement. First, any improvement of the ratio $\frac{n'}{|S|}$ (currently $3$) would immediately significantly improve the factor. Second, with the arguments in this paper we can prove that the edges of any traceable graph with \(n\) vertices can be separated by \((6+2/3)n\) paths. In fact, any improvement of that could be plugged in the proof and used to improve the bound.
In contrast, we observe that Theorem~\ref{thm:dean-kouider} is tight for disjoint triangles.

We also note that by using the \(M_k\)'s together with Lemma~\ref{lemma:brandt}, we can prove that the edges of any graph can be covered by \(6n\) paths without using Theorem~\ref{thm:dean-kouider}. This allows us to verify Conjecture~\ref{conj:lin-sep}
in a self-contained manner.

  As a final remark, we observe that robust sublinear expanders,
which play a key role in~\cite{letzter2022separating}, have also
been applied in~\cite{bucic2022towards} to investigate the well-known Erd\H os--Gallai
conjecture, that says that the edges of any $n$-vertex graph can
be partitioned into $\bigoh(n) $ edges and
cycles. It is natural to wonder if our
approach might be useful in tackling this conjecture as well.

\section*{Acknowledgments}

The first author expresses gratitude to Bhargav Narayanan for sharing the question with her at the Oberwolfach Workshop id:2201, and to the organizers and participants of that workshop for creating a stimulating environment.
We also thank Matija Buci\'c, Victor Falgas-Ravry and Lyuben Lichev for insightful comments which enriched this manuscript.

\bibliographystyle{amsplain}


\begin{aicauthors}
\begin{authorinfo}[bonamy]
  Marthe Bonamy\\
  Univ. Bordeaux, CNRS, Bordeaux INP, LaBRI, UMR 5800, F-33400\\
  Talence, France\\
  marthe.bonamy\imageat{}u-bordeaux\imagedot{}fr \\
  \url{https://www.labri.fr/perso/mbonamy/}
\end{authorinfo}
\begin{authorinfo}[botler]
  F\'abio Botler\\
  Professor\\
  Programa de Engenharia de Sistemas e Computa\c c\~ao \\
	Instituto Alberto Luiz Coimbra de P\'os-Gradua\c c\~ao e Pesquisa em Engenharia \\
	Universidade Federal do Rio de Janeiro \\
	Rio de Janeiro, Brasil\\
  fbotler\imageat{}cos\imagedot{}ufrj\imagedot{}br \\
  \url{http://www.cos.ufrj.br/~fbotler}
\end{authorinfo}
\begin{authorinfo}[dross]
  François Dross\\
  Associate Professor \\
  Univ. Bordeaux, CNRS, Bordeaux INP, LaBRI, UMR 5800, F-33400\\
  Talence, France\\
  francois.dross\imageat{}u-bordeaux\imagedot{}fr \\
  \url{https://sites.google.com/view/francoisdross/}
\end{authorinfo}
\begin{authorinfo}[naia]
  T\'assio Naia\\
  María de Maetzu fellow \\
  Centre de Reserca Matemática \\
  tnaia\imageat{}member\imagedot{}fsf\imagedot{}org \\
  \url{https://tassio.gitlab.io/site/}
\end{authorinfo}
\begin{authorinfo}[skokan]
  Jozef Skokan\\
  Professor\\
  Department of Mathematics\\
  London School of Economics and Political Science (LSE) \\
  Houghton Street, London, WC2A 2AE, United Kingdom\\
  j.skokan\imageat{}lse\imagedot{}ac\imagedot{}uk \\
  \url{http://www.maths.lse.ac.uk/personal/jozef/}
\end{authorinfo}
\end{aicauthors}

\end{document}